\documentclass[letterpaper,11pt]{amsart}
\usepackage{amsmath,amssymb,amsthm,pinlabel,tikz,hyperref,mathrsfs,thmtools}
\usepackage{verbatim}
\usepackage{float}

\usetikzlibrary{patterns}

\newcommand{\nc}{\newcommand}
\nc{\dmo}{\DeclareMathOperator}
\dmo{\ra}{\rightarrow}
\dmo{\Z}{\mathbb{Z}}
\dmo{\R}{\mathbb{R}}
\dmo{\T}{\mathcal{T}}
\dmo{\C}{\mathcal{C}}
\dmo{\inte}{int }
\dmo{\N}{\mathcal{N}}
\dmo{\AC}{\mathcal{AC}}
\dmo{\Mod}{Mod}
\dmo{\PMod}{PMod}
\dmo{\B}{B}
\dmo{\bH}{\mathbb{H}}
\dmo{\PB}{PB}
\dmo{\Sp}{Sp}
\dmo{\I}{\mathcal{I}}
\dmo{\el}{\ell_{\C}}
\dmo{\NN}{\mathcal{N}}
\dmo{\Tr}{Tr}
\dmo{\llangle}{\langle\langle}
\dmo{\rrangle}{\rangle\rangle}

\usepackage{tikz-cd}

\tikzcdset{row sep/normal=1.5cm}
\tikzcdset{column sep/normal=1.5cm}

\usepackage{caption}
\usepackage{subcaption}

\usetikzlibrary{decorations.markings}
\tikzset{->-/.style={decoration={
			markings,
			mark=at position #1 with {\arrow{>}}},postaction={decorate}}}

\nc{\nt}{\newtheorem}

\nt{theorem}{Theorem}
\newtheorem{thm}{{\bf Theorem}}[section]
\newtheorem{lem}[thm]{{\bf Lemma}}
\newtheorem{cor}[thm]{{\bf Corollary}}

\newtheorem{remark}[thm]{Remark}

\newtheorem{question}[thm]{Question}

\newtheorem{definition}[thm]{Definition}
\numberwithin{equation}{section}

\title[Reducible normal generators]{Reducible normal generators for\\mapping class groups are abundant}

\date{\today}
\author{Hyungryul Baik}
\address{
	Department of Mathematical Sciences, KAIST\\
	291 Daehak-ro Yuseong-gu, Daejeon, 34141, South Korea 
}
\email{%
	hrbaik@kaist.ac.kr
}

\author{Dongryul M. Kim}

\address{%
	Department of Mathematics, Yale University\\
	219 Prospect Street, New Haven, CT 06511, USA
}
\email{%
	dongryul.kim@yale.edu
}

\author{Chenxi Wu}

\address{%
	Department of Mathematics, University of Wisconsin--Madison\\
	480 Lincoln Drive, Madison, WI 53706, USA
}
\email{%
	cwu367@math.wisc.edu
}

\begin{document}
\begin{abstract}
	In this article, we study the normal generation of the mapping class group. We first show that a mapping class is a normal generator if its restriction on the invariant subsurface normally generates the (pure) mapping class group of the subsurface. As an application, we provided a criterion for reducible mapping classes to normally generate the mapping class groups in terms of its asymptotic translation lengths on Teichm\"uller spaces. This is an analogue to the work of Lanier-Margalit dealing with pseudo-Anosov normal generators. 
\end{abstract}

\maketitle

%%%%%%%%%%%%%%%%%%%%%%%%%%%%%%%%%%%%%%%%%%%%%%%%%%%%%%%%%%%%%%%%%%%%%%
%
%	Introduction
%
%%%%%%%%%%%%%%%%%%%%%%%%%%%%%%%%%%%%%%%%%%%%%%%%%%%%%%%%%%%%%%%%%%%%%%

\section{Introduction}	\label{section:introduction}

Let $S = S_g$ be a closed connected orientable surface with genus $g > 1$ and let $\Mod(S)$ be its mapping class group. From \cite{lickorish1964finite} and \cite{mumford1967abelian}, it is well-known that $\Mod(S)$ can be normally generated by a Dehn twist along a non-separating simple closed curve (simply, curve). It is natural to ask what other mapping classes normally generate the mapping class group and how abundant they are to study the algebraic structure of $\Mod(S)$.

There are several ways to interpret the abundancy of normal generators. For instance, we can consider a random walk on $\Mod(S)$, and then see whether the random walk produces normal generators with high probability or not. In this sense, Maher and Tiozzo showed in \cite{maher1807random} that a random mapping class is not a normal generator by showing that the normal closure of a random element is a free group (see also \cite{dahmani2017hyperbolically} and \cite{clay2021right} for normal generators for proper subgroups of $\Mod(S)$). 

Since a random mapping class is pseudo-Anosov (see \cite{kowalski2008sieve}, \cite{rivin2008walks}, \cite{maher1807random}, \cite{baik2022topological}, \cite{baik2021linear} for instance), one might interpret the result of Maher and Tiozzo as most pseudo-Anosov mapping classes are not normal generators. However, there are abundant examples of pseudo-Anosov normal generators, and one general construction is given by Lanier and Margalit \cite{lanier2022normal}. Indeed, for $f \in \Mod(S)$, Lanier and Margalit studied translation length of $f$ on the Teichm\"uller space $\T(S)$, which can be considered in two different ways; the first one is more standard one, i.e., $\tau_{\T}(f) := \inf\limits_{x \in \T(S)} d_{\T}(x, f(x))$. We will call this quantity \emph{minimal translation length}. The second one is so-called \emph{asymptotic translation length} (a.k.a. \emph{stable translation length}) which is defined by $\ell_{\T}(f) := \lim_{n\to\infty} \dfrac{d_{\T}(x, f^n(x))}{n}$ for any $x \in \T(S)$. Then Lanier and Margalit showed that a pseudo-Anosov element $f$ with small minimal translation length $\tau_{\T}(f)$ normally generates the mapping class group where ``small" means less than $\frac{1}{2} \log 2$.

In this perspective, it is natural to ask the same question about non-pseudo-Anosov mapping classes. Here is our main question: 
\begin{question} \label{que:normalgeneration} 
	If a reducible element $f$ of $\Mod(S)$ has small translation length on $\T(S)$, does it normally generate $\Mod(S)$ except for obvious counterexamples? 
\end{question}

The question is of course open-ended since it does not specify what we mean by ``small" and ``obvious counterexample". For the latter, we note that any Dehn twist has zero asymptotic translation length on $\T(S)$. On the other hand, the Dehn twist along a separating curve (in short, we call it a separating twist) cannot normally generate $\Mod(S)$ while the Dehn twist along a non-separating curve (in short, a non-separating twist) does normally generate $\Mod(S)$. Separating twists are obvious counterexamples. One can define ``small" in the question as in the theorem of Lanier-Margalit. Hence, a formalized version of Question \ref{que:normalgeneration} is as follows:

\begin{question} \label{que:normalgenerationmodified} 
	For a reducible element $f$ of $\Mod(S)$ with $0 < \ell_{\T}(f) \le \frac{1}{2} \log 2$, does it normally generate $\Mod(S)$? 

\end{question}

In this paper, we give a partial answer to Question \ref{que:normalgenerationmodified}. We call a mapping class $f \in \Mod(S)$ \emph{partly pseudo-Anosov} if there is an embedded subsurface $A \subseteq S$ which is invariant under $f$ and $f|_A$ is pseudo-Anosov. Then we prove:

\begin{restatable*}{thm}{smalltrans} \label{thm:maintheorem}
	Let $S$ be a closed connected orientable surface and $f \in \Mod(S)$ be partly pseudo-Anosov with invariant subsurface $A$ on which $f|_A$ is pseudo-Anosov. If $\ell_{\T}(f) \le \frac{1}{2} \log 2$ and $A$ has genus at least three, then $f$ is a normal generator of $\Mod(S)$.
\end{restatable*}

This generalizes Lanier-Margalit's result in the sense that the philosophy ``small translation length implies normal generation" works for a strictly larger collection of mapping classes than pseudo-Anosov. Indeed, every reducible mapping class with positive asymptotic translation length on the Teichm\"uller space has a power which is partly pseudo-Anosov. This observation also deduces the partial answer for Question \ref{que:normalgeneration} from Theorem \ref{thm:maintheorem}.

\begin{restatable*}{cor}{transred} \label{cor:transred}
	For $f \in \Mod(S_g)$ with $g > 2$, $0 < \ell_{\T}(f) \le \frac{\log 2}{2(2g-2)}$ implies that the normal closure of $f$ contains the commutator subgroup of $\Mod(S_g)$.
	
	In particular, let $f \in \Mod(S_g)$ be a mapping class with a maximal invariant multicurve $\gamma$ such that each component of $S_g \setminus \gamma$ has genus at least three. If $0 < \ell_{\T}(f) \le \frac{3\log 2}{2(g-1)}$, then $f$ normally generates $\Mod(S_g)$.
\end{restatable*}

In \cite{tsai2009asymptotic}, Tsai figured out the asymptotic behavior of $\ell_{g,n} := \inf \{ \ell_{\T}(f) : f \in \Mod(S_{g, n}) \mbox{ is pseudo-Anosov}\}$ where $S_{g, n}$ is a surface of genus $g$ with $n$ punctures. Namely, for any fixed $g \geq 2$, Tsai showed that there exists a constant $c_g$ such that $$ \dfrac{\log n}{c_g n} \leq \ell_{g, n} \leq \dfrac{c_g \log n}{n}. $$ 
At the first glance, one might think this result is an obstruction for reducible elements satisfying the assumptions in Corollary \ref{cor:transred} to exist. However, the dependence of the constant $c_g$ on $g$ is not explicit and a priori it is not clear at all if such reducible elements indeed exist. Hence, we leave this as the following question.

\begin{question}
	Is there a reducible $f \in \Mod(S_g)$ with $0 < \ell_{\T}(f) \le \frac{\log 2}{2(2g-2)}$?
\end{question}

Finally we note that in Theorem \ref{thm:maintheorem}, we interpreted the translation length in Question \ref{que:normalgeneration} as the asymptotic translation length. Since in general the minimal translation length is bigger than or equal to the asymptotic translation length by definition, now the result holds with the minimal translation length as well.  In the case of pseudo-Anosov mapping classes, the distinction between the asymptotic translation length and minimal translation length disappears, since any pseudo-Anosov element $f$ has a unique invariant axis in $\T(S)$ on which $f$ acts by a translation.

Theorem \ref{thm:maintheorem} follows from the locality of the normal generation together with Minsky's product region theorem for Teichm\"uller space \cite{minsky1996extremal}. More precisely, we also prove the following.

\begin{restatable*}{thm}{locality} \label{thm.main1}
	Let $S$ be a closed connected orientable surface and let $A \subseteq S$ be an embedded subsurface with genus at least three. If $f \in \Mod(S)$ leaves $A$ invariant and the normal closure of $f|_A$ contains the commutator subgroup of $\PMod(A)$, then $f$ is also a normal generator of $\Mod(S)$.

	Moreover, if the normal closure of $f|_A$ contains $\PMod(A)$, then the normal closure of $f$ contains the commutator subgroup of $\Mod(S)$ even when $A$ has genus one or two. In particular, when $S$ has genus at least three, then $f$ is a normal generator of $\Mod(S)$.
\end{restatable*}

Based on this locality property, we could provide a systematic way to produce reducible normal generators. Indeed, in \cite{baik2023asymptotic}, Baik, Kin, Shin, and Wu observed that normal generators are generic in a fibered cone of a hyperbolic 3-manifold fibering over the circle. Here, the genericity is considered in the sense that there is a translation of the fibered cone in which every primitive element has monodromy that normally generates the mapping class group. As an application of Theorem \ref{thm.main1}, we proved the genericity for more general class of fibered 3-manifolds in Section \ref{sec:generic}.

In \cite{baik2023asymptotic}, it is conjectured \cite[Conjecture 1.3]{baik2023asymptotic} that all but finitely many primitive elements in the fibered cone of a hyperbolic fibered 3-manifold give normal generators of mapping class groups, which originates from Margalit's question \cite{margp}. The authors answered affirmatively when the fibered cone is of dimension two. Our results in Section \ref{sec:generic} also give a partial answer to a non-hyperbolic version of this conjecture, while essential tools for hyperbolic version may not be directly applied to this variant.

\begin{remark}
	Theorems so far can be easily seen when we consider a mapping class which is identity outside of the invariant subsurface and restricted to a normal generator on the subsurface. This is because conjugation of the map on the subsurface can be extended to a conjugation that is the identity outside the subsurface, together with the fact that a non-separating twist is a normal generator.

	However, it is nontrivial when we consider general reducible mapping classes, for instance, partly pseudo-Anosov maps.
\end{remark}

\subsection*{Acknowledgements}

The authors greatly appreciate Ian Biringer, Ethan Farber, Changsub Kim, Yair N. Minsky, and Dragomir Saric for helpful discussions. We would like to thank anonymous referees for helpful comments.
The first author was partially supported by the National Research Foundation of Korea(NRF) grant funded by the Korea government(MSIT) (No. 2020R1C1C1A01006912).

%%%%%%%%%%%%%%%%%%%%%%%%%%%%%%%%%%%%%%%%%%%%%%%%%%%%%%%%%%%%%%%%%%%%%%
%
%	Normal generation of pA
%
%%%%%%%%%%%%%%%%%%%%%%%%%%%%%%%%%%%%%%%%%%%%%%%%%%%%%%%%%%%%%%%%%%%%%%

\section{Previously known results on Normal generation of $\Mod(S)$} \label{sec:normal}

In this section, we recall a result of Lanier-Margalit. They showed that a pseudo-Anosov mapping class with small asymptotic translation length on Teichm\"uller space normally generates the mapping class group. More precisely:

\begin{thm}[{\cite[Theorem 1.2]{lanier2022normal}}] \label{thm.LMtrans}
	Let $S$ be a connected orientable surface of genus $g\ge 0$ and with finitely many punctures. If a pseudo-Anosov $f \in \Mod(S)$ satisfies $\ell_{\T}(f) \le \frac{1}{2} \log 2$, then the normal closure of $f$ contains the commutator subgroup of $\Mod(S)$. In particular, if $g \ge 3$, then $\ell_{\T}(f) \le \frac{1}{2} \log 2$ implies that $f$ normally generate $\Mod(S)$ ($\PMod(S)$ if the surface is punctured).
\end{thm}

We note here that while Lanier and Margalit considered closed surfaces, it holds for surfaces with finitely many punctures. Indeed, the proof of Theorem \ref{thm.LMtrans} is based on the so-called well-suited curve criteria and the combinatorics of a systole with respect to the singular Euclidean metric on the surface induced by a pseudo-Anosov. One can observe that both work for surfaces with finitely many punctures. Indeed, \cite[Section 3]{lanier2022normal} remarks the well-suited curve criteria for such surfaces, and the exact same proofs of \cite[Lemma 2.5, Proposition 2.7]{farb2008lower} work for surfaces with finitely many punctures.

Lanier and Margalit also provided a general version of the well suited curve criteria, which involves certain graphs encoding the combinatorics of curves on a surface and dynamics of a mapping class.

\begin{definition}
	Let $S$ be a surface and $f \in \Mod(S)$. $\N_f(S)$ is a graph whose vertices are isotopy classes of non-separating curves on $S$, and two vertices $\alpha$, $\beta$ are connected by an edge if there is a conjugation $hfh^{-1}$, $h \in \Mod(S)$, such that $\beta = hfh^{-1}(\alpha)$.
\end{definition}

Note that $\N_f(S)$ is not a subgraph of a curve complex $\C(S)$ in general. Lanier and Margalit showed the following necessary and sufficient condition for a mapping class to be a normal generator.

\begin{thm}[{\cite[Proposition 7.3]{lanier2022normal}}] \label{thm.LM}
	Let $S$ be a closed connected orientable surface. Then the normal closure of $f \in \Mod(S)$ contains the commutator subgroup if and only if $\mathcal{N}_f(S)$ is connected. In particular, if $S$ has genus at least three, $\N_f(S)$ is connected if and only if $f$ normally generates $\Mod(S)$.
\end{thm}

%%%%%%%%%%%%%%%%%%%%%%%%%%%%%%%%%%%%%%%%%%%%%%%%%%%%%%%%%%%%%%%%%%%%%%
%
%	Normal generation of reducible
%
%%%%%%%%%%%%%%%%%%%%%%%%%%%%%%%%%%%%%%%%%%%%%%%%%%%%%%%%%%%%%%%%%%%%%%

\section{Locality of normal generation}

In this section, we prove the normal generation of the mapping class group is a local property:

\locality

By Masur and Misnky \cite{masur1999geometry}, any two vertices in the curve complex are connected by a path. However, $\mathcal{N}_f(S)$ has different features from the curve complex. First, all vertices of $\mathcal{N}_f(S)$ correspond to non-separating curves. Furthermore, an edge of $\mathcal{N}_f(S)$ may not be an edge of the curve complex, and vice versa. As such, our strategy is modifying a path in the curve complex into a path in $\mathcal{N}_f(S)$ via appropriate surgery. We start with the following well-known fact \cite[Theorem 4.4]{FarbMargalit12}. We record the proof for the sake of completeness.

\begin{lem} \label{lem.nonsep}
	Let $\alpha, \beta \in \C(S)$ be vertices corresponding to non-separating curves. Then there is a path $\eta$ in $\C(S)$ connecting $\alpha$ and $\beta$ consisting of vertices corresponding to non-separating curves.
\end{lem}

\begin{proof}
	Let $\xi$ be a path in $\C(S)$ that connects $\alpha$ and $\beta$. Such $\xi$ exists due to the connectedness of $\C(S)$. Let $x$, $y$, $z$ be consecutive three vertices in $\xi$, and we denote the subsegment of $\xi$ connecting $x$, $y$, and $z$ by $[x, y, z]$.

	Suppose that $y$ is separating. Then there are two possibilities;\begin{enumerate}
		\item $x$ and $z$ are contained in the same component of $S \setminus y$; In this case, let $y'$ be a non-separating curve in the component of $S \setminus y$ that does not contain $x$ and $z$. Then replace $[x, y, z]$ with $[x, y', z]$.
		\item $x$ and $z$ are contained in different components of $S \setminus y$; In this case, $x$ and $z$ are disjoint so we replace $[x, y, z]$ with $[x, z]$.
	\end{enumerate}
	
	Since $\alpha$ and $\beta$ are non-separating, we can conduct this kind of surgery at each vertex of $\xi$ whenever it corresponds to a separating curve. Then as a consequence, we get a new path $\eta$ in $\C(S)$ connecting $\alpha$ and $\beta$, whose vertices corresponding to non-separating curves.
\end{proof}

We can modify further so that none of the consecutive vertices in the path is a bounding pair.

\begin{lem} \label{lem.nonbound}
	Let $\alpha, \beta \in \C(S)$ be vertices corresponding to non-separating curves. Then there is a path $\gamma$ in $\C(S)$ connecting $\alpha$ and $\beta$ consisting of vertices corresponding to non-separating curves, and none of two consecutive vertices in $\gamma$ forms a bounding pair of $S$.
\end{lem}

\begin{proof}
	By Lemma \ref{lem.nonsep}, we have a path $\eta$ in $\C(S)$ connecting $\alpha$ and $\beta$, and consisting solely of non-separating curves. Let $x$ and $y$ be consecutive vertices in $\eta$, and suppose that they form a bounding pair. Then there is a non-separating curve $z$ in $S \setminus (x \cup y)$ so that neither $x$ and $z$ nor $y$ and $z$ forms a bounding pair. Then we can replace $[x, y]$ with $[x, z, y]$ to obtain a new path in $\C(S)$. Continuing this surgery whenever two consecutive vertices form a bounding pair, we finally obtain a path $\gamma$ in $\C(S)$ that none of two consecutive vertices form a bounding pair. It is straightforward that $\gamma$ consists of non-separating curves and endpoints of $\gamma$ are $\alpha$ and $\beta$.
\end{proof}

\subsection*{Proof of Theorem \ref{thm.main1}}
First we consider the case when $A$ has genus at least three. Let $\alpha$, $\beta$ be vertices of $\mathcal{N}_f(S)$. By Lemma \ref{lem.nonsep} and Lemma \ref{lem.nonbound}, there is a path $\gamma$ in $\C(S)$ connecting $\alpha$ and $\beta$, with the property that \begin{itemize}
	\item every vertex in $\gamma$ is non-separating and;
	\item each two consecutive vertices in $\gamma$ do not form a bounding pair.
\end{itemize} Hence, in order to show that $\alpha$ and $\beta$ is connected by a path in $\mathcal{N}_f(S)$, it suffices to consider the case when $\alpha$ and $\beta$ are disjoint curves that do not form a bounding pair.

Denote the genus of $S$ by $g$. Since $S \setminus (\alpha \cup \beta)$ is of genus $g-2$ with four boundary components, there exists $\varphi \in \Mod(S)$ that $\varphi(\alpha)$ and $\varphi(\beta)$ are in $A$. Since there exists a pure mapping class of $A$ that maps $\varphi(\alpha)$ to $\varphi(\beta)$ and $\PMod(A)$ is perfect \cite[Theorem 5.2]{FarbMargalit12}, the hypothesis that the commutator subgroup of $\PMod(A)$ is contained in the normal closure of $f|_{A}$ implies that there exists $f_1, \ldots, f_n \in \Mod(A)$ satisfying $$\varphi(\beta) = (f_1^{-1} f|_A f_1) \circ \cdots \circ (f_n^{-1} f|_A f_n)(\varphi(\alpha)).$$
Extending $f_i$'s to the entire surface $S$, we get $$\beta = (\varphi^{-1}f_1^{-1} f f_1 \varphi) \circ \cdots \circ (\varphi^{-1}f_n^{-1} f f_n \varphi)(\alpha).$$ It concludes that $\alpha$ and $\beta$ are connected by a path in $\mathcal{N}_f(S)$. Consequently, the normal closure of $f$ contains the commutator subgroup of $\Mod(S)$ by Theorem \ref{thm.LM}, and thus $f$ is the normal generator of $\Mod(S)$ due to the perfectness of $\Mod(S)$ \cite[Theorem 5.2]{FarbMargalit12}.

Now to see the second assertion, suppose that the normal closure of $f|_A$ contains $\PMod(A)$. When $A$ has genus at least two, the second assertion follows from the fact that the perfectness of $\PMod(A)$ no longer play a role in the above argument. However, when $A$ has genus one, we cannot find $\varphi \in \Mod(S)$ that $\varphi(\alpha)$ and $\varphi(\beta)$ are in $A$. Hence, we introduce a new curve which are connected with $\alpha$ and $\beta$ in $\N_f(S)$.

Again, noting that $S \setminus (\alpha \cup \beta)$ is of genus $g-2$ with four boundary components, we can find a non-separating curve $\delta$ in $S$ that $i(\alpha, \delta) = i(\beta, \delta) = 1$. Since a tubular neighborhood of $\alpha \cup \delta$ is a subsurface homeomorphic to one-holed torus, there exists $\varphi \in \Mod(S)$ that $\varphi(\alpha)$ and $\varphi(\delta)$ are contained in $A$. Hence, by the same argument above, $\alpha$ and $\delta$ are connected by a path in $\N_f(S)$. Similarly, $\delta$ and $\beta$ are so. It completes the proof.

\section{Mapping classes with small asymptotic translation lengths}

As in Theorem \ref{thm.LMtrans}, a pseudo-Anosov $f \in \Mod(S_g)$ is a normal generator if its asymptotic translation length on Teichm\"uller space is small. One natural question in this line of thought is whether the similar criterion can be provided without assumption that the mapping class is pseudo-Anosov. We come up with a collection of mapping classes which is strictly larger than the collection of pseudo-Anosovs.

\begin{definition}[partly pseudo-Anosov mapping class]
	Let $S$ be a closed connected orientable surface of genus $g \ge 2$. $f \in \Mod(S_g)$ is called \emph{partly pseudo-Anosov} if there exist a representative $f_0$ of $f$ and an embedded subsurface $A \subseteq S_g$ so that $f_0|_A$ is a pseudo-Anosov homeomorphism of $A$. This is same to say that $\tau_{\T}(f) > 0$. We simply denote by $f|_A$ the mapping class of $A$ represented by $f_0|_A$.
\end{definition}

\begin{remark}
	There is a similar notion, called \emph{partial pseudo-Anosov} mapping classes (see \cite{minsky2021bottlenecks}). $f \in \Mod(S)$ is called partial pseudo-Anosov when there is a subsurface $A \subsetneq S$ that $f|_A$ is pseudo-Anosov and $f|_{S \setminus A}$ is the identity. Hence, partial pseudo-Anosov mapping classes are partly pseudo-Anosov, but the converse does not hold. In particular, pseudo-Anosov mapping class are partly pseudo-Anosov while they are not partial pseudo-Anosov.
\end{remark}

Then we state a suitable generalization of the criterion of Lanier and Margalit.

\smalltrans

The key ingredient of the proof is Minsky's product region theorem for Teichm\"uller spaces. The Teichm\"uller space $\T(S)$ can be decomposed into a product of $\T(S \setminus \partial A)$ and $\prod_{|\partial A|} \bH ^2$. Minsky proved the following theorem comparing distances in $\T(S)$ and $\T(S \setminus \partial A) \times \prod_{|\partial A|} \bH^2$.

\begin{thm}[Product region theorem, {\cite{minsky1996extremal}}] \label{thm.minsky}
	Let $\gamma$ be a multicurve on $S$, and $X = \T(S\setminus \gamma) \times \prod_{|\gamma|} \bH^2$, equipped with the metric $$d_X = \max \{d_{\T(S\setminus \gamma)}, d_{\bH^2}, \ldots, d_{\bH^2}\}.$$ Then the Fenchel-Nielsen coordinates on $\T(S)$ give rise to a natural homeomorphism $\Pi : \T(S) \to X$. Futhermore, for $\epsilon > 0$ sufficiently small, there exists $c = c(S, \epsilon)$ such that $$|d_{\T}(\sigma, \tau) - d_X(\Pi(\sigma), \Pi(\tau))| \le c$$ for any $\sigma$, $\tau \in \T(S)$ at which $\gamma$ has length less than $\epsilon$.
\end{thm}

\begin{proof}[Proof of Theorem \ref{thm:maintheorem}]

	Now fix a point $\sigma \in \T(S)$ that $\partial A$ has length less than $\epsilon$.  Since $A$ is invariant under $f$, $\partial A$ also has its length less than $\epsilon$ at $f \sigma$. Then from Theorem \ref{thm.minsky}, we have $$d_{\T}(\sigma, f^n \sigma) \ge d_{\T(A\setminus \partial A)}(\pi(\sigma), f|_A^n \pi(\sigma)) - c$$ where $\pi(\sigma)$ is the $\T(A \setminus \partial A)$ component of $\Pi(\sigma)$. Consequently, $\ell_{\T}(f) \le \frac{1}{2} \log 2$ implies $\ell_{\T(A \setminus \partial A)}(f|_A) \le \frac{1}{2} \log 2$. Then by Theorem \ref{thm.LMtrans}, the commutator subgroup of $\PMod(A)$ is contained in the normal closure of $f|_A$, and therefore applying Theorem \ref{thm.main1} completes the proof.
\end{proof}

Further observation on reducible mapping classes gives an analogous criterion for them as in \cite{lanier2022normal}. Let $f \in \Mod(S)$ be reducible with maximal invariant multicurve $\gamma$ and $\ell_{\T}(f) > 0$. Then some power of $f$ leaves each component of $S \setminus \gamma$ and $\gamma$ invariant. If the power is periodic on each component of $S \setminus \gamma$, then we can take more power so that it becomes a multitwist. As in \cite{thurston1988geometry}, the action of a multitwist on the Teichm\"uller space can be restricted to an isometrically embedded hyperbolic plane in the Teichm\"uller space. Since the restricted action is parabolic, it contradicts to $\ell_{\T}(f) > 0$. Hence, the power should be partly pseudo-Anosov.

Now let $A$ be a component of $S \setminus \gamma$. Then $A$ is isotopic to at least one of $f(A), \ldots, f^N(A)$ by pigeonhole principle where $N$ is the number of components in $S \setminus \gamma$. Together with this observation, Theorem \ref{thm:maintheorem} implies the following.

\transred

\section{Genericity of normal generation in a fibered cone} \label{sec:generic}

In this section, we consider 3-manifolds fibering over the circle. Explicitly, let $f \in \Mod(S)$ be a mapping class, and consider its mapping torus $M_f$. Denote by $[f] \in H^1(M_f;\R)$ the image of the generator of $H^1(S^1;\Z)$ via the corresponding fibration $M_f \to S^1$. Then there exists a closed cone $\C_f$ in $H^1(M_f;\R)$ containing $[f]$ so that every primitive integral class in its interior corresponds to each fibration over the circle \cite{thurston1986norm}. This closed cone $\C_f$ is called \emph{fibered cone}.
Thurston also defined a norm $\| \cdot \|$ on $H^1(M_f;\R)$, so-called Thurston norm, to study fibrations of $M_f$ by investigating $H^1(M_f;\R)$.

When $f$ is pseudo-Anosov, then the mapping torus $M_f$ is hyperbolic \cite{thurston1998hyperbolic} and the work of McMullen \cite[Corollary 5.4]{McMullen00} on Teichm\"uller polynomial provides the necessary decay of $\ell_{\T}(\varphi)$ for primitive integral $[\varphi] \in \inte \C_f$ as $\lVert [\varphi] \rVert \to \infty$. Based on this observation, Baik, Kin, Shin, and Wu proved the following normal generation phenomenon in the fibered cone:

\begin{thm}[{\cite[Theorem 3.4]{baik2023asymptotic}}]
Let $f \in \Mod(S)$ be pseudo-Anosov and consider its mapping torus $M_f$ with fibered cone $\C_f$.	Let $L$ be a 2-dimensional rational subspace of $H^1(M;\R)$. Then for all but finitely many primitive integral classes $[\varphi] \in \C_f$, $\varphi$ normally generates $\Mod(S_{\varphi})$ where $S_{\varphi}$ is the fiber surface of the fibration corresponding to $[\varphi]$.
\end{thm}

Our purpose in this section is to generalize this result to non-hyperbolic fibered 3-manifolds. To do this, let $f \in \Mod(S)$ be a partly pseudo-Anosov mapping class and let $A$ be its invariant subsurface on which $f|_A$ is pseudo-Anosov. Then the inclusion $A \hookrightarrow S$ induces the inclusion $i : M_{f|_A} \hookrightarrow M_f$ of the hyperbolic fibered 3-manifold $M_{f|_A}$. Now we have the homomorphism $i^* : H^1(M_f;\R) \to H^1(M_{f|_A}; \R)$. Then we have the following.

\begin{thm}
	Let $f \in \Mod(S)$ be partly pseudo-Anosov and consider the mapping torus $M_f$ with fibered cone $\C_f$. Then there exists $x \in \C_f$ such that every primitive integral class $[\varphi] \in x + \C_f$ normally generates $\Mod(S_{\varphi})$.
\end{thm}

\begin{proof}
	One can observe that $i^* : H^1(M_{f|_A};\R) \to H^1(M_f;\R)$ maps the fibered cone $\C_f$ into the fibered cone $\C_{f|_A}$ with $i*([f]) = [f|_A]$. In particular, it is a non-trivial map on $\C_f$. Hence, for $x \in \C_f$ with $i^*(x) \neq 0$, we can make the distance between $i^*(x + \C_f)$ and $\partial \C_{f|_A}$ large enough by taking large $x$. Then as in \cite{baik2023asymptotic}, we can take $x$ so large that decreasing behavior of $\ell_{\T}$ in the fibered cone implies that $\ell_{\T}(\varphi) \le \frac{1}{2}\log 2$ for all $[\varphi] \in i^*(x + \C_f)$. Now Theorem \ref{thm.LMtrans} and Theorem \ref{thm.main1} completes the proof.
\end{proof}

\begin{remark}
	Since $M_f$ is not hyperbolic when $f$ is not pseudo-Anosov, it is non-trivial whether we have such decaying phenomenon of $\ell_{\T}$ on $\C_f$ in general. This is the reason why the proof in \cite{baik2023asymptotic} may not be directly applied to this case.
\end{remark}

Now let $L$ be a rational subspace of $H^1(M_f; \R)$ such that $i^*(L)$ is of dimension at most 2. Applying the argument in \cite{baik2023asymptotic}, we deduce the following.

\begin{thm}
	Let $f \in \Mod(S)$ be partly pseudo-Anosov and consider the mapping torus $M_f$ with fibered cone $\C_f$. Let $L$ be a rational subspace of $ H^1(M_f;\R)$ with $\dim i^*(L) \le 2$. For all but finitely many primitive integral classes $(S_{\varphi}, \varphi)$ in $\C_f \cap L$, $\varphi$ normally generates $\Mod(S_{\varphi})$.
\end{thm}

One direct consequence of this theorem is a partial answer to \cite[Conjecture 1.3]{baik2023asymptotic}.

\begin{cor}
	Let $f \in \Mod(S)$ be partly pseudo-Anosov with invariant subsurface $A$ that $f|_A$ is pseudo-Anosov. If the dimension of maximal invariant subspace of $H_1(A; \R)$ under the action of $f|_A$ is at most one, then all but finitely many primitive integral classes $(S_{\varphi}, \varphi)$ in the fibered cone $\C_f$ normally generates $\Mod(S_{\varphi})$.
\end{cor}

%%%%%%%%%%%%%%%%%%%%%%%%%%%%%%%%%%%%%%%%%%%%%%%%%%%%%%%%%%%%%%%%%%%%%%%%%%%%%%%
%
%	References
%
%%%%%%%%%%%%%%%%%%%%%%%%%%%%%%%%%%%%%%%%%%%%%%%%%%%%%%%%%%%%%%%%%%%%%%%%%%%%%%%

\medskip
\bibliographystyle{alpha} 
\bibliography{reducible}

\begin{thebibliography}{BKSW23}

\bibitem[BCK21]{baik2021linear}
Hyungryul Baik, Inhyeok Choi, and Dongryul~M Kim.
\newblock Linear growth of translation lengths of random isometries on gromov
  hyperbolic spaces and teichm\"uller spaces.
\newblock {\em arXiv preprint arXiv:2103.13616}, 2021.
\newblock To appear in J. Inst. Math. Jussieu.

\bibitem[BCK22]{baik2022topological}
Hyungryul Baik, Inhyeok Choi, and Dongryul~M Kim.
\newblock Topological entropy of pseudo-anosov maps from a typical thurston’s
  construction.
\newblock {\em International Mathematics Research Notices},
  2022(24):19862--19904, 2022.

\bibitem[BKSW23]{baik2023asymptotic}
Hyungryul Baik, Eiko Kin, Hyunshik Shin, and Chenxi Wu.
\newblock Asymptotic translation lengths and normal generation for
  pseudo-anosov monodromies of fibered 3--manifolds.
\newblock {\em Algebraic \& Geometric Topology}, 23(3):1363--1398, 2023.

\bibitem[CMM21]{clay2021right}
Matt Clay, Johanna Mangahas, and Dan Margalit.
\newblock Right-angled {A}rtin groups as normal subgroups of mapping class
  groups.
\newblock {\em Compositio Mathematica}, 157(8):1807--1852, 2021.

\bibitem[DGO17]{dahmani2017hyperbolically}
Fran{\c{c}}ois Dahmani, Vincent Guirardel, and Denis Osin.
\newblock {\em Hyperbolically embedded subgroups and rotating families in
  groups acting on hyperbolic spaces}, volume 245.
\newblock American Mathematical Society, 2017.

\bibitem[FLM08]{farb2008lower}
Benson Farb, Christopher~J Leininger, and Dan Margalit.
\newblock The lower central series and pseudo-anosov dilatations.
\newblock {\em American journal of mathematics}, 130(3):799--827, 2008.

\bibitem[FM12]{FarbMargalit12}
Benson Farb and Dan Margalit.
\newblock {\em A primer on mapping class groups}, volume~49 of {\em Princeton
  Mathematical Series}.
\newblock Princeton University Press, Princeton, NJ, 2012.

\bibitem[Kow08]{kowalski2008sieve}
Emmanuel Kowalski.
\newblock {\em The large sieve and its applications: Arithmetic Geometry,
  Random Walks and Discrete Groups}, volume 175 of {\em Cambridge Tracts in
  Mathematics}.
\newblock Cambridge University Press, 2008.

\bibitem[Lic64]{lickorish1964finite}
William~BR Lickorish.
\newblock A finite set of generators for the homeotopy group of a 2-manifold.
\newblock In {\em Mathematical Proceedings of the Cambridge Philosophical
  Society}, volume~60, pages 769--778. Cambridge University Press, 1964.

\bibitem[LM22]{lanier2022normal}
Justin Lanier and Dan Margalit.
\newblock Normal generators for mapping class groups are abundant.
\newblock {\em Commentarii Mathematici Helvetici}, 97(1):1--59, 2022.

\bibitem[Mar]{margp}
Dan Margalit.
\newblock personal communication.

\bibitem[McM00]{McMullen00}
Curtis~T. McMullen.
\newblock Polynomial invariants for fibered 3-manifolds and {T}eichm\"uller
  geodesics for foliations.
\newblock {\em Ann. Sci. \'Ecole Norm. Sup. (4)}, 33(4):519--560, 2000.

\bibitem[Min96]{minsky1996extremal}
Yair~N Minsky.
\newblock Extremal length estimates and product regions in teichm{\"u}ller
  space.
\newblock {\em Duke Mathematical Journal}, 83(2):249--286, 1996.

\bibitem[MM99]{masur1999geometry}
Howard~A Masur and Yair~N Minsky.
\newblock Geometry of the complex of curves i: Hyperbolicity.
\newblock {\em Inventiones mathematicae}, 138(1):103--149, 1999.

\bibitem[MM21]{minsky2021bottlenecks}
Yair Minsky and Babak Modami.
\newblock Bottlenecks for weil-petersson geodesics.
\newblock {\em Advances in Mathematics}, 381:107628, 2021.

\bibitem[MT07]{maher1807random}
Joseph Maher and Giulio Tiozzo.
\newblock Random walks, wpd actions, and the cremona group.
\newblock {\em Proceedings of the London Mathematical Society}, 1807.

\bibitem[Mum67]{mumford1967abelian}
David Mumford.
\newblock Abelian quotients of the teichm{\"u}ller modular group.
\newblock {\em Journal d’Analyse Math{\'e}matique}, 18(1):227--244, 1967.

\bibitem[Riv08]{rivin2008walks}
Igor Rivin.
\newblock Walks on groups, counting reducible matrices, polynomials, and
  surface and free group automorphisms.
\newblock {\em Duke Mathematical Journal}, 142(2):353--379, 2008.

\bibitem[Thu86]{thurston1986norm}
William~P Thurston.
\newblock A norm for the homology of 3-manifolds.
\newblock {\em Mem. Amer. Math. Soc.}, 59:99--130, 1986.

\bibitem[Thu88]{thurston1988geometry}
William~P Thurston.
\newblock On the geometry and dynamics of diffeomorphisms of surfaces.
\newblock {\em Bulletin (new series) of the american mathematical society},
  19(2):417--431, 1988.

\bibitem[Thu98]{thurston1998hyperbolic}
William~P Thurston.
\newblock Hyperbolic structures on 3-manifolds, ii: Surface groups and
  3-manifolds which fiber over the circle.
\newblock {\em arXiv preprint math/9801045}, 1998.

\bibitem[Tsa09]{tsai2009asymptotic}
Chia-Yen Tsai.
\newblock The asymptotic behavior of least pseudo-{A}nosov dilatations.
\newblock {\em Geometry \& Topology}, 13(4):2253--2278, 2009.

\end{thebibliography}

\end{document}